\documentclass[12pt,reqno]{amsart}
\usepackage{amsaddr}
\usepackage{graphicx}
\usepackage{setspace}
\usepackage{empheq}
\usepackage{cite}
\usepackage{textcomp}
\usepackage{dutchcal}
\usepackage{amssymb}
\usepackage{hyperref}
\usepackage{enumerate}

\usepackage{color}
\definecolor{MyLinkColor}{rgb}{0,0,0.4}

\newcommand{\Gkh}{\mathcal{G}_{kh}}
\newcommand{\Ckh}{\mathcal{C}_{kh}}
\newcommand{\Dlt}{\bigtriangleup}

\newtheorem{thm}{Theorem}[section]

\newtheorem{definition}[thm]{Definition}
\newtheorem{lemma}[thm]{Lemma}

\newtheorem{remark}[thm]{Remark}

\numberwithin{equation}{section}   

\title{Capillary-Gravity Water Waves: Modified Flow Force Formulation}

\author[C.~I.~Martin]{Calin Iulian Martin}
\address{Faculty of Mathematics, University of Vienna, Oskar Morgenstern Platz 1, 1090 Vienna, Austria \\
	E-mail address: \href{mailto:calin.martin@univie.ac.at}{\textsf{calin.martin@univie.ac.at}}}

\author[B. Basu]{Biswajit Basu$^*$}
\thanks{$^*$Corresponding author}
\address{School of Engineering, Trinity College Dublin, Dublin 2, Ireland.\\
	E-mail address: \href{mailto:basub@tcd.ie}{\textsf{basub@tcd.ie}}}

\begin{document}

\maketitle
\begin{abstract}
\noindent The classical irrotational capillary-gravity water wave problem described by the Euler equations with a nonlinear free surface boundary condition over a flat bed is considered. A 
modified flow force has been defined and a new formulation of capillary-gravity waves in the framework of the modified flow force function has been developed. Using bifurcation theory, the local existence of waves of small amplitude is proved.

\vspace{1em}
\noindent
{\bfseries Keywords}: Capillary-gravity, water waves, flow force, free surface flows, finite depth, local bifurcation.\\
 {\bfseries Mathematics Subject Classification}: 35Q35, 35J15

\end{abstract}

\section{Introduction}
Research on periodic traveling water waves though initiated in 1847 by Stokes \cite{stokes1847}, has been still ongoing with contributions from Nekrasov, Levi-Civita and Struik \cite{nekrasov1921steady, levi1924determinazione, struik1926determination} in 1920s and from  Krasovskii, Keady and Norbury, Toland and McLeod  \cite{krasovskii1961theory, toland1996stokes,  mcleod1980stokes} subsequently.
More recently, the existence of large amplitude water waves with vorticity was proved rigorously in the paper by Constantin and Strauss \cite{constantin2004exact}. 
For previous notable research in this area we refer the reader to \cite{constantin2011periodic, walsh2009steady, walsh2009stratified, ac11, conswahlen07, 
constantin2004symmetry, constantin2006variational, constantin2007stability, constantin2004symdeep, ac11am,
cons_varvaruca2011}.

In the previous studies, the stream function formulation has been one of the popular approaches to deal with the water wave problem. In this formulation, the mass flux of the flow is assumed to be a constant.  First in the paper by Benjamin and Lighthill \cite{benjamin}, and then in the paper by Keady and Norbury \cite{keadynorbury78b} a quantity called flow force related to the flow of the fluid with a free surface over a flat bed has been mentioned. This quantity has the property that it is  zero on the flow bed and is invariant on the free surface. The latter property allows
the transformation of the unknown domain (with a free surface) into a fixed domain by means of a suitable change of variables. Also, the level curves of the flow force function within the fluid domain are monotone in the vertical direction. For values of the Bernoulli\rq{}s constant close to the critical one, Kozlov \cite{kozlov} used the flow force to uniquely parameterize the subset of waves with crest located on a fixed vertical and verified the Benjamin-Lighthill conjecture. Using the concept of ``flow force'', recently Basu \cite{Basu2019JDE} has proposed a novel flow force formulation of the free surface nonlinear boundary value problem for water waves. The formulation by Basu \cite{Basu2019JDE} for the case of gravity water waves needs to be suitably modified to account also for surface tension effects. Indeed, surface tension is present in gravity driven free surface flows and some recent research results on capillary-gravity water waves can be found in 
\cite{HenCG, henrymatioc, martin2013, MarMatSJAM, MarMatJNS, MarMatJDE, MatMfM, MatMatCMP, wahlen2006}.

The contribution of this paper is to develop a flow force based formulation of the irrotational capillary-gravity water waves with free surface flow over a flat bed. In this context, a ``modified flow force'' function is proposed to handle the capillary-gravity case. While we assume the free surface to be a graph, we do not need to require a priori the absence of stagnation points in the water flow. However, we consider irrotational flows in this paper and irrotational flows do not have stagnation points.
The original water wave problem is reformulated as a quasilinear pseudodifferential equation for a function of one variable, giving the elevation of free surface when the fluid domain is the conformal image of a half-plane. Instrumental in achieving our goals is the \emph{Dirichlet-Neumann operator}. 
The method of Crandall-Rabinowitz of bifurcation from simple eigenvalues is used to prove the existence of small amplitude capillary-gravity waves.

\section{Presentation of the water wave problem}
\subsection{Stream function formulation of the water wave problem}
Let us consider a water wave profile oscillating about a flat surface. The problem of two-dimensional periodic steady waves traveling at speed $c$, has a 
space-time dependence of the free surface, of the velocity field and of the pressure of the form ($\widetilde{X}-ct$) and is periodic with period $L>0$. The flat bed is given 
by $\widetilde{Y}=0$ and the oscillating wave profile is represented by $\widetilde{Y}=\eta (\widetilde{X}-ct)$, where $\widetilde{Y}$ is the vertical axis. The free boundary problem can be simplified by
a change of the reference frame from a moving one to a stationary one. Indeed, the transformation $(\widetilde{X}-ct, \widetilde{Y}) \mapsto (X,Y)$ allows us to supress the time variable from the equations of motion. To introduce the latter we denote the water domain by $\Omega$ which is assumed to be bounded below by the impermeable flat bed
    $$\mathcal{B} = \{(X,0); X \in \mathbb{R}\},$$
    and above by an a priori unknown curve
    \begin{equation}\label{shape_freesurface}
    \mathcal{S} = \{(X,\eta(X));X \in \mathbb{R}\},
    \end{equation}
    satisfying
    $$\eta(X+L)=\eta(X)\quad{\rm for}\,\,{\rm all}\quad X\in\mathbb{R},$$
    where $L>0$ is a constant that denotes the period of the motion.
The motion of the water flow obeys the Euler equations
\begin{equation}
\begin{split}
(u_1-c)u_{1X}+u_2 u_{1Y}&=-P_X,\\
(u_1-c)u_{2X}+u_2 u_{2Y}&=-P_Y-g,
\end{split}
\end{equation}
and the equation of mass conservation
\begin{equation}\label{mass_cons}
u_{1X}+u_{2Y}=0,
\end{equation}
where $(u_1, u_2)$ and $P$  represent the velocity field and the pressure in the flow, respectively, and are sought to be periodic of period $L$. The description of the water wave problem is complete together with the specification of the boundary conditions. 
Of them we start with the kinematic boundary conditions that ensure that the two boundaries are impermeable. More precisely, the top boundary is defined by the 
condition
\begin{equation}\label{kin_surf}
u_2(X,\eta(X))=(u_1(X,\eta(X))-c)\eta_X
\end{equation} 
while, on the bottom, we require no motion in the vertical direction, that is
\begin{equation}
 u_2(X,0)=0\quad{\rm for}\,\,{\rm all}\,\,X\in\mathbb{R}.
\end{equation}
We close the system by the dynamic boundary condition
\begin{equation}\label{dyn_bc}
P(X, \eta(X)) = P_{atm} - \sigma \frac{\eta_{XX}}{(1+\eta_{X}^{2})^{3/2}} ~~\text{on}~~Y=\eta(X),
\end{equation}
where $P_{atm}$ is the constant atmospheric pressure and $\sigma>0$ denotes the coefficient of surface tension.

Aiming at reducing the number of unknowns, we introduce the stream function $(X,Y)\rightarrow \psi(X,Y)$ which
 is defined (up to an additive constant) by the relations
\begin{equation}\label{stream_func}
    \psi_X=-u_2, \quad \psi_Y=u_1-c \ .
\end{equation}
We assume throughout the paper that the flow is irrotational, that is 
\begin{equation}\label{eq:eq2}
u_{1_Y}-u_{2_X}\equiv 0,
\end{equation}
equation that readily implies via \eqref{stream_func} that $\psi$ is a harmonic function.

Another characteristic of the flow is the mass flux (relative to the uniform flow $c$) which is defined as
\begin{equation}\label{eq:eq3}
p_0=\int_{0}^{\eta(X)} (u_1(X,Y)-c) dY.
\end{equation}
Availing of the equation of mass conservation \eqref{mass_cons} and of \eqref{kin_surf} we infer that $p_0$ is indeed a constant.


From the Euler equation, we get Bernoulli's law which states that the quantity
\begin{equation}\label{eq:eq5}
E := \frac{(c-u_{1})^2+v^2}{2}+gY+P 
\end{equation}
is a constant (called hydraulic head) throughout the flow.
The dynamic boundary condition at the surface can be restated as
\begin{equation}\label{eq:eq6}
    \psi^2_X+\psi_Y^2+2gY - 2\sigma \frac{\eta_{XX}}{(1+\eta_{X}^{2})^{3/2}}=Q~~\text{on}~~Y=\eta(X)
\end{equation}
where $Q=2(E-P_{atm})$ is a constant for any given flow.

We can formulate the free surface irrotational capillary-gravity water wave problem as the system \cite{constantin2004exact}
\begin{equation}\label{syst_psi}
    \begin{split}
        \Dlt\psi= 0 ~~\text{in}~~0<Y<\eta(X),\\
        |\nabla\psi|^2+2gY - 2\sigma \frac{\eta_{XX}}{(1+\eta_{X}^{2})^{3/2}}=Q~~\text{on}~~Y=\eta(X),\\
        \psi=0~~\text{on}~~Y=\eta(X),\\
        \psi=-p_0~~\text{on}~~Y=0,
    \end{split}
\end{equation}
which is to be solved for functions which are $L$-periodic in the $X-$variable within the fluid domain $ \Omega$.

\subsection{Modified flow force reformulation}

We now reformulate the governing equations in Section 2 using a different and novel approach
based on the modified flow force function.  
We recall that the flow force function (in the absence of surface tension) is given by the expression
\begin{equation} \label{flowforce}
\overline{S}(X,Y) =  \int_{0}^Y [P(X,r) + ( u_{1}(X,r) - c)^2] {\rm d}r  \ \ 
\end{equation}
where,
the pressure at $(x,y)$ is denoted by the function $P(x,y)$  
and $P_{atm}$ is the atmospheric pressure. For our purposes (which pertain to a problem involving the combined effect of gravity and surface tension) we will need a modified flow force function. The latter is defined as

\begin{equation} \label{mod_flowforce}
S(X,Y) =  \overline{S}(X,Y) +e(X,Y) \ \ 
\end{equation}
where
\begin{equation} \label{e}
e(X,Y) =e_{0}(X) \frac{Y}{\eta(X)} \ ,  
\end{equation}
and
\begin{equation} \label{e_0}
e_{0}(X) = - P_{atm} \eta(X) - \frac{ \sigma }{\sqrt{{1+\eta_{X}^{2}}}}  + \sigma 
\end{equation}

\begin{remark}
An aspect of vital importance in the following considerations is 
that the (modified) flow force function $S$ is constant on the two boundaries of the fluid domain. Indeed, it is immediate from \eqref{flowforce} and \eqref{e} that 

\begin{align*}
\begin{array}{ll}
S = 0 & \mbox{ on the bottom }\mathcal{B}. \\
\end{array}
\end{align*}
We will show in the following that $S$ is constant on the free surface. 
To this end we note that the gradient of $\overline{S}$ satisfies 
\begin{equation}\label{gradS}
 \overline{S}_X=-(u_1-c)u_2,\qquad \overline{S}_Y=P(X,Y)+(u_1(X,Y)-c)^2,
\end{equation}
where the first equation above follows from the equality
\begin{equation}\label{eq:12}
     P_X + 2(u_{1}-c) {u_{1}}_X
    = -(u_{1}-c) {u_{2}}_Y - u_{2}{u_{1}}_Y = -\{(u_{1}- c)u_{2}\}_Y ,
    \end{equation}
which is a consequence of the Euler- and mass conservation equation.

Note also that on the free surface $\mathcal{S}$ we have
\begin{equation}
e(X,\eta(X)) =e_0(X).
\end{equation}
Setting now $S_0(X):=S(X,\eta(X))$ we find that 

\begin{equation}\label{eq:11}
    \begin{split}
    \frac{dS_0}{dX} =&\eta_X \Big(P(X,\eta(X)) + (u_{1}(X,\eta(X)) - c)^2\Big) \\
    &-u_{2}(X,\eta(X))\Big(u_{1}(X,\eta(X))-c \Big)\\ 
    & - \eta_X\left(P_{atm}-\sigma\frac{\eta_{XX}}{   (1+\eta_{X}^{2})^{\frac{3}{2}}    }\right)\\
     = & (u_1-c)\big( (u_1-c) \eta_X -u_2 \big)\\
     = & 0,
    \end{split}
\end{equation}
where, to obtain the last equalities above, we have also used the dynamic and surface kinematic conditions. Thus, $S_0$ is, in fact, a constant.
\end{remark}

Recalling that 

\begin{equation}\label{eq:17}
\begin{split}
&S_X = (c-u_{1})u_{2} +e_{X}, \\
&S_Y  = P+ (u_{1}-c)^2 + e_{Y}, 
\end{split}
\end{equation}
we differentiate the first equation with respect to $X$ and the second one with respect to $Y$, obtaining that 

\begin{equation}\label{eq:18}
S_{XX} + S_{YY} =- g - 2(c-u_{1}) ({u_1}_Y - {u_{2}}_X) +e_{XX} 
\end{equation}
which, in the absence of any vorticity, leads to 

\begin{equation}\label{eq:19}
\Delta S = -g + e_{XX} 
\end{equation}
with $S = 0$ on the flat bed $\mathcal{B}$ and $S = S_0$ on the free surface $\mathcal{S}$. The nonlinear dynamic boundary condition on the free surface may be expressed as
\begin{equation}\label{eq:20}
 \left(\frac{\big(S_X - e_X)^2}{(S_Y - P - e_Y)}
+ (S_Y - P - e_Y)\right)\Big|_{(X, \eta(X))} 
+2 g Y -  2 \sigma\frac{\eta_{XX}}{(1+\eta_{X}^{2})^{\frac{3}{2}} }= Q.
\end{equation}
Summing up, we have proved that the water wave problem \eqref{syst_psi} can be written in terms of the modified flow force function $S$ as
\begin{equation}\label{syst_S}
 \begin{split}
  \Delta S = -g + e_{XX},\\
  S=S_0\,\,{\rm on}\,\,\mathcal{S},\\
  S=0\,\,{\rm on}\,\,\mathcal{B},\\
  \left(\frac{\big(S_X - e_X)^2}{(S_Y - P - e_Y)}
+ (S_Y - P - e_Y)\right)\Big|_{(X, \eta(X))} 
+2 g Y -  2 \sigma\frac{\eta_{XX}}{(1+\eta_{X}^{2})^{\frac{3}{2}} }= Q.
 \end{split}
\end{equation}

\subsection{Formulation of the water wave problem as a quasilinear equation for the free surface}
We will aim to further reformulate the problem \eqref{syst_S} as a single equation for the free surface $\mathcal{S}$. To this end, we need a few notations and some preliminary results. For the proof of these results we refer the 
reader to \cite{cons_varvaruca2011}.
\begin{definition}
For $p\in\mathbb{N}$ and $\alpha\in (0,1)$ we denote by $C^{p,\alpha}$ the space of Hölder continuous functions. 
The space of functions of class $C^{p,\alpha}$ over any compact subset of their domain of definition will be denoted with $C_{\textrm{loc}}^{p,\alpha}$. Moreover, by $C_{2\pi}^{p,\alpha}$ we denote the space of functions of one real variable which are $2\pi$ periodic and of class $C_{\textrm{loc}}^{p,\alpha}$ in $\mathbb{R}$, while
by $C_{2\pi, o}^{p,\alpha}$ we denote the functions that are in $C_{2\pi}^{p,\alpha}$ and have zero mean over one period. Finally, by $L_{2\pi}^{2}$ we denote the space of $2\pi$-periodic locally square integrable functions of one real variable. By $L_{2\pi ,\textrm{o}}^{2}$ we denote the subspace of $L_{2\pi}^{2}$ whose elements have zero mean over 
one period.
\end{definition}
A key ingredient for completing our tasks is played by the \emph{Dirichlet-Neumann operator} associated to a horizontal strip in the plane. Before we introduce the Dirichlet-Neumann operator let us set
$$\mathcal{R}_{d}:=\{(x,y)\in\mathbb{R}^{2}:-d<y<0\},$$ where $d>0$.
\begin{definition}
Given $w\in C_{2\pi}^{p,\alpha}$ we let $W\in C^{p,\alpha}(\overline{\mathcal{R}}_{d})$ be the unique solution of
\begin{equation}
 \begin{array}{rcl}
\Delta W & = & 0\,\,\textrm{in}\,\,\mathcal{R}_{d},\\
W(x,-d) & =  & 0,\,\, x\in\mathbb{R},\\
W(x,0) & = & w(x),\,\,x\in\mathbb{R}.
\end{array}
\end{equation}
The function $(x,y)\rightarrow W(x,y)$ is $2\pi$-periodic in $x$ throughout $\mathcal{R}_{d}$. For $p\in\mathbb{Z},p\geq 1$, and $\alpha\in (0,1)$  we define the \textit{periodic 
Dirichlet-Neumann operator for a strip} $\mathcal{G}_{d}$ by 
$$\mathcal{G}_{d}(w)(x):=W_{y}(x,0),\,\,x\in\mathbb{R}.$$
\end{definition}
\noindent The Dirichlet-Neumann operator possesses the following properties.
\begin{lemma}\label{DN_prop}
\begin{enumerate}
\item [ $(i)$ ] The operator $\mathcal{G}_{d}:C_{2\pi}^{p,\alpha}\rightarrow C_{2\pi}^{p-1,\alpha}$ is a bounded linear operator.
\item [ $(ii)$ ] If the function $w$ takes the constant value $c$ then 
\begin{equation}
 \label{DNonconst}
\mathcal{G}_{d}(c)=\frac{c}{d}.
\end{equation}
\end{enumerate}
\end{lemma}
Closely related to the Dirichlet-Neumann operator is the periodic Hilbert transform for the strip $\mathcal{R}_d$. More precisely, having the previous notation in mind and setting $Z$ to be the unique (up to a constant) harmonic function such that $Z+iW$ is holomorphic in $\mathcal{R}_d$, we define for $w\in C_{2\pi, o}^{p,\alpha}$ 
$$\mathcal{C}_d(w)(x):=Z(x,0)\,\,{\rm for}\,\,x\in\mathbb{R}.$$ The map $\mathcal{C}_d$ is called the \emph{periodic Hilbert transform} for the strip $\mathcal{R}_d$. Some useful properties of the Hilbert transform are listed below.
\begin{lemma}\label{ht_rep}
Assume that 
$$w=\sum_{n=1}^{\infty}a_{n}\cos (nx)+\sum_{n=1}^{\infty}b_{n}\sin (nx),$$ is the Fourier series expansion of $w\in L_{2\pi,\textrm{o}}^{2}$.
 Then
\begin{equation}
 \mathcal{C}_{d}(w)=\sum_{n=1}^{\infty}a_{n}\coth (nd)\sin (nx)-\sum_{n=1}^{\infty}b_{n}\coth (nd)\cos (nx)
\end{equation}
\end{lemma}
\begin{lemma}\label{ht_prop}
\begin{enumerate}
\item [ $(i)$ ] For all $d>0, \, p\in\mathbb{N}$ and $\alpha\in (0,1)$ it holds that $\mathcal{C}_d: C_{2\pi, o}^{p,\alpha}\rightarrow C_{2\pi, o}^{p,\alpha}$ is a bounded linear operator. Moreover, $\mathcal{C}_{d}^{-1}=-\mathcal{C}_{d}:C_{2\pi,\textrm{o}}^{p,\alpha}\rightarrow C_{2\pi,\textrm{o}}^{p,\alpha}$ is also a bounded linear operator.
\item [ $(ii)$ ] $\mathcal{G}_{d}(w)=(\mathcal{C}_{d}(w))^{\prime}=\mathcal{C}_{d}(w^{\prime})$  {\rm for all}
$w\in  C_{2\pi, o}^{p,\alpha}$ mit $p\geq 1$.
\item  [ $(iii)$ ] If $w\in  C_{2\pi}^{p,\alpha}$ then $\mathcal{G}_{d}(w)=\frac{\left[w\right]}{d}+\mathcal{C}_{d}(w^{\prime})$, where $[w]$ denotes the average of $w$ over one period.
\end{enumerate}
\end{lemma}
\noindent The following two definitions specify the type of solutions we will be proving to exist for the free boundary value problem \eqref{syst_S}.
\begin{definition}
 We say that a solution $(\Omega, S)$ of the water wave problem \eqref{syst_S} is of class $C^{2,\alpha}$ if the free surface belongs to $C^{2,\alpha}_{2\pi}$ and $S\in C^2(\Omega)\cap C^{2,\alpha}(\overline{\Omega})$.
\end{definition}
\begin{definition}
\begin{itemize}
\item 
We say that $\Omega\subset\mathbb{R}^{2}$ is an \textit{L-periodic strip like domain} if it is contained in the upper half $(X,Y)$-plane and if its boudary consists of the real 
axis $\mathcal{B}$ and of a parametric curve $\mathcal{S}$ defined by $\eqref{shape_freesurface}$ which is $L$-periodic.
\item For any such domain, the \textit{conformal mean depth} is defined to be unique positive number $h$ such that there exists an onto conformal mapping 
$\tilde{U}+i\tilde{V}:\mathcal{R}_{h}\rightarrow\Omega$ which admits an extension between the closures of these domains, with onto mappings
$$\{(x,0):x\in\mathbb{R}\}\rightarrow \mathcal{S},$$
and $$\{(x,-h):x\in\mathbb{R}\}\rightarrow \mathcal{B},$$ and such that 

\begin{equation}
\label{}
\begin{array}{lll}
\tilde{U}(x+L,y) & = &\tilde{U}(x,y)+L\\
\tilde{V}(x+L,y) & = &\tilde{V}(x,y)
 \end{array}\,\,\,\,{\rm for}\,\,{\rm all}\,\, (x,y)\in\mathcal{R}_{h}.
\end{equation}
\end{itemize}
\end{definition}
The existence and uniqueness of such an $h$ was proved in Appendix A of the paper \cite{cons_varvaruca2011}.

We are now ready to state the main result of the current section, namely the reformulation of the water wave problem as
a quasilinear equation for a periodic function (representing the elevation of the free surface) of one variable in a fixed domain.
\begin{thm}
        If $(\Omega, S)$ of class $C^{2,\alpha}$ is a solution of \eqref{syst_S} then there exists a positive number h, a function $v\in C^{2,\alpha}_{2\pi}$ and a constant a $\in \mathbb{R}$ such that
        \begin{equation}\label{eq_v}
            \begin{aligned}
                &\frac{(A\Gkh(v) - Bv')^2}{(\Gkh(v)^2+v'^2)\left(A v'+B\Gkh(v) - \left(P_{atm} - \sigma\frac{\Gkh(v)v''-\Gkh(v')v'}{(\Gkh(v)^2+v'^2)^{3/2}}\right)(\Gkh(v)^2+v'^2)\right)} \\
                &\\
                &+ \frac{ Av'+B\Gkh(v)}{\Gkh(v)^2+v'^2}  -  P_{atm} + \sigma \frac{\Gkh(v)v''-\Gkh(v')(v')}{(\Gkh(v)^2+v'^2)^{3/2}}     \\
                &\\
                &=
                \left(Q + 2\sigma \frac{\Gkh(v)v'' - \Gkh(v')v'}{(v'^2+\Gkh(v)^2)^{3/2}} -2gv \right) 
                \end{aligned}
                \end{equation}
                where
               \begin{equation}
               \begin{split}
                &A= P_{atm}v'+\sigma \frac{v'(\Gkh(v')v' - \Gkh(v)v'')  }{(\Gkh(v)^2+v'^2)^{3/2}} \\
                &B=  \frac{S_0-\sigma}{kh}+P_{atm}\Gkh(v)+g(\Gkh(\frac{v^2}{2})-v\Gkh(v))+\sigma \Gkh \left(\frac{\Gkh(v)}{(\Gkh(v)^2+v'^2)^{1/2}}   \right)\\
                &[v]=h \\
                &v(x)>0\ for\ all\ x\ \in \ \mathbb{R}, \\
                &the\ mapping\ x \rightarrow \left(\frac{x}{k} + \Ckh(v - h)(x), v(x) \right)\ is\ injective\ on\ \mathbb{R}, \\
                &v'(x)^2 + \Gkh(v)(x)^2 \neq 0\ for\ all\ x\ \in  \ \mathbb{R} ,
            \end{split}
        \end{equation}
        where $k = \frac{2\pi}{L}$. Moreover
        \begin{equation}\label{eq_Surface}
            \begin{aligned}
            &\mathcal{S} = \left\{\left(a + \frac{x}{k} + \Ckh(v - h)(x),v(x) \right):x\in \mathbb{R}\right\}.
            \end{aligned}
        \end{equation}
        Conversely, let $h>0$ and $v \in C^{2,\alpha}_{2\pi}$ be such that \eqref{eq_v} holds. Assume also that $\mathcal{S}$ is defined by \eqref{eq_Surface}, 
        let $\Omega$ be the domain whose boundary consists of $\mathcal{S}$ and of the real axis $\mathcal{B}$ and let $a\in\mathbb{R}$ be arbitrary. Then there exists a function $S$ in $\Omega$ such that $(\Omega,S)$ 
        is a solution of \eqref{syst_S} of class $C^{2,\alpha}$.
    \end{thm}
    \begin{proof} 
    We first prove the necessity. Let $(\Omega,S)$ be a solution of class $C^{2,\alpha}$ of \eqref{syst_S}. Then we denote by $h$ the conformal mean depth of $\Omega$ and by $\tilde{U}+i\tilde{V}$ the conformal mapping associated to $\Omega$. If we consider the mapping $U+iV:\mathcal{R}_{kh} \rightarrow \Omega$ given by
    \begin{equation}\label{UV_rescaling}
        \begin{aligned}
            U(x,y) = \tilde{U} \left(\frac{x}{k},\frac{y}{k}\right) \\
            V(x,y) = \tilde{V} \left(\frac{x}{k},\frac{y}{k}\right)
        \end{aligned}, \quad (x,y)\in \mathcal{R}_{kh},
    \end{equation}
    where $k = \frac{2\pi}{L}$, then following the proof of Theorem 2.2 in \cite{cons_varvaruca2011} we see that $U,V  \in C^{2,\alpha}(\overline{R_h})$ and $U + iV$ is a conformal mapping from $\mathcal{R}_{kh}$ onto $\Omega$ which extends homeomorphically to the closures of these domains, with onto mappings
    $$\{(x,0):x \in\mathbb{R}\} \rightarrow \mathcal{S},$$
    and
    $$\{(x,-kh): x \in \mathbb{R}\} \rightarrow \mathcal{B}.$$
    Moreover,
    $$U_x^2(x,0) + V_x^2(x,0) \neq 0 \,\,{\rm for}\,\, {\rm all}\,\,  x \in \mathbb{R}. $$
    For $x\in\mathbb{R}$ we set 
    $$v(x): = V(x,0), \quad u(x): = U(x,0) .$$
   Let us remark that the set $\{(u(x),v(x)):\,x\in\mathbb{R}\}$ constitutes  a reparametrization of the free surface $Y=\eta(X)$. Consequently, we have for all $x\in\mathbb{R}$
\begin{equation}\label{change_var}
\begin{split}
v^{\prime}(x)&=u^{\prime}(x)\eta_X(u(x)),\\
v^{\prime\prime}(x)&=u^{\prime\prime}(x)\eta_X(u(x))+(u^{\prime}(x))^2\eta_{XX}(u(x)).
\end{split}
\end{equation}
Moreover, it holds that
    $$u = \Ckh(v)$$
    and 
    \begin{equation}\label{uv_rel}
    u' = \Gkh(v) \ and \ u'' = \Gkh(v'). 
    \end{equation}
    It also follows \cite{cons_varvaruca2011} that $v\in  C_{2\pi}^{2,\alpha}$ and
    $$[v]=h, \ \ \ \ $$
    $$v(x) > 0 \,\,{\rm  for}\,\, {\rm all}\,\, x\in\mathbb{R},  $$
    $${\rm the}\,\,{\rm map}\,\, x\rightarrow \left(\frac{x}{k} + \Ckh(v - h)(x), v(x)    \right)\,\,{\rm is}\,\,{\rm injective}\,\, {\rm on}\,\, \mathbb{R},$$
    \begin{equation}\label{surface_form}
    \mathcal{S}=\left \{ \left(a + \frac{x}{k} + \Ckh(v - h)(x),v(x)     \right) : x \in  \mathbb{R} \right \}, 
    \end{equation}
    for some $a \in \mathbb{R}$, whose presence in formula \eqref{surface_form} is due to the invariance of problem \eqref{syst_S} to horizontal translations. From \eqref{uv_rel} and the Cauchy-Riemann equations it follows that 
    $$v'(x)^2 + \mathcal{G}_{kh}(v)(x)^2 \neq 0\,\,{\rm for}\,\,{\rm all}\,\, x\in \mathbb{R}. $$
    
    Now let $\xi: \mathcal{R}_{kh} \rightarrow \mathbb{R}$ be defined by
    \begin{equation}\label{xi}
    \xi(x,y) = S(U(x,y),V(x,y))-e(U(x,y),V(x,y)),\ (x,y) \in\mathcal{R}_{kh}.
    \end{equation}
    The harmonicity in $\Omega$ of the function $(X,Y)\rightarrow S(X,Y)+\frac{g}{2}Y^2 - e(X,Y)$ and the invariance of harmonic functions under conformal mappings imply that 
    $$\xi + \frac{g}{2}V^2 \ \text{is \ harmonic \ in} \ \mathcal{R}_{kh}.$$
    The chain rule and the Cauchy-Riemann equations imply that 
    \begin{equation}\label{CR_outcome}
        \begin{pmatrix}
        S_X- e_X \\
        S_Y - e_Y
        \end{pmatrix}\Big|_{(U(x,y),V(x,y))}
        = \frac{1}{v'^2+\mathcal{G}_{kh}(v)^2}
        \begin{pmatrix}
        V_y & -V_x \\
        V_x & V_y
        \end{pmatrix}
        \begin{pmatrix}
        \xi_x \\
        \xi_y
        \end{pmatrix}
        \ in \ \overline{\mathcal{R}_{kh}}
    \end{equation}
    Define $\zeta:\mathcal{R}_{kh}\rightarrow \mathbb{R}$ through
    \begin{equation}\label{zeta}
    \zeta (x,y)= \xi (x,y)+\frac{g}{2}V^2(x,y).
    \end{equation}
    Using the boundary conditions we obtain that $\zeta$ verifies the system
    \begin{equation}\label{zeta_eq}
    \begin{split}
    \bigtriangleup \zeta &= 0 \ in \ \mathcal{R}_{kh}\\
    \zeta(x,-kh)&=0 \,\,{\rm for}\,\,{\rm all}\,\, x\in \mathbb{R}\\
   \zeta(x,0)& = S_0 - e(u(x),v(x)) + \frac{g}{2}v^2\,\,{\rm for}\,\,{\rm all}\,\, x\in  \mathbb{R},
   \end{split}
   \end{equation}
   from which we infer that $$\zeta_y(x,0)=\mathcal{G}_{kh}\left( S_0 - e(u(x),v(x)) + \frac{g}{2}v^2\right).$$
    Note that, by \eqref{change_var}, we have
    \begin{equation}
    \begin{split}
    e(u(x),v(x))=e_0(u(x))\frac{v(x)}{\eta(u(x))}
    =-P_{atm}v(x)-\sigma\frac{u^{\prime}(x)}{\sqrt{   (u^{\prime}(x))^2 + (v^{\prime}(x))^2   }}+\sigma,
    \end{split}
    \end{equation}
    which implies
   \begin{equation}
    \begin{split}
    \xi (x,0)&=S(U(x,0),V(x,0))-e(U(x,0),V(x,0))=S(u(x),v(x))-e(u(x),v(x))\\
    &=S_0-\sigma+P_{atm}v(x)+\sigma\frac{u^{\prime}(x)}{\sqrt{   (u^{\prime}(x))^2 + (v^{\prime}(x))^2   }}.
    \end{split}
    \end{equation}
    Therefore,
    \begin{equation*}
        \begin{aligned}
            &\xi_x (x,0) = P_{atm}v'+\sigma\frac{v^{\prime}(u^{\prime\prime}v^{\prime}-u^{\prime}v^{\prime\prime} ) } 
             {(u^{\prime 2}+v^{\prime 2})^{\frac{3}{2}}}  = P_{atm}v' +\sigma  \frac{v'\left(  \Gkh(v')v'-\Gkh(v)v''\right)}{(\Gkh(v)^2+v'^2)^{3/2}}=:A \\
            &\xi_y (x,0) = \zeta_y(x,0) - g(VV_y)(x,0) = \frac{S_0-\sigma}{kh}+g(\Gkh(\frac{v^2}{2})-v\Gkh(v)) +P_{atm} \Gkh(v) \\
            &\qquad \qquad+\sigma \Gkh \left(\frac{\Gkh(v)}{(\Gkh(v)^2+v'^2)^{1/2}}   \right) =:B
        \end{aligned}
    \end{equation*}
    Hence, the last equation of \eqref{syst_S} can be rewritten as
   \begin{equation}
            \begin{aligned}
                &\frac{(A\Gkh(v) - Bv')^2}{(\Gkh(v)^2+v'^2)\left(A v'+B\Gkh(v) - \left(P_{atm} - \sigma\frac{\Gkh(v)v''-\Gkh(v')v'}{(\Gkh(v)^2+v'^2)^{3/2}}\right)(\Gkh(v)^2+v'^2)\right)} \\
                &\\
                &+ \frac{ Av'+B\Gkh(v)}{\Gkh(v)^2+v'^2}  -  P_{atm} + \sigma \frac{\Gkh(v)v''-\Gkh(v')(v')}{(\Gkh(v)^2+v'^2)^{3/2}}     \\
                &\\
                &=
                \left(Q + 2\sigma \frac{\Gkh(v)v'' - \Gkh(v')v'}{(v'^2+\Gkh(v)^2)^{3/2}} -2gv \right), 
                \end{aligned}
                \end{equation}
                that is \eqref{eq_v} holds true.
    For the sufficiency suppose that the positive number $h$ and the function $v \in  C_{2\pi}^{2,\alpha}$ satisfy \eqref{eq_v}. Let $V$ be the harmonic function in $\mathcal{R}_{kh}$ which satisfies
    $$V(x,-kh)=0,$$
    and
    $$V(x,0) = v(x) \,\,{\rm for}\,\, {\rm all}\,\, x \in \mathbb{R},$$
    and let $U:\mathcal{R}_{kh} \rightarrow \Omega$ be such that $U+iV$ is holomorphic. An application of Lemma 2.1 from \cite{cons_varvaruca2011} yields that $U+iV \ \in \ C^{2,\alpha}( \overline{\mathcal{R}}_{kh})$. From $[v]=h$ we obtain
    \begin{equation}\label{UV}
        \left \{
        \begin{aligned}
            &U(x+2\pi,y)=U(x,y)+\frac{2\pi}{k}\\
            &V(x+2\pi,y)=V(x,y)
        \end{aligned}\,\,\,\, {\rm for}\,\,(x,y) \ \in \ \mathcal{R}_{kh}.
        \right .
    \end{equation}
    From the injectivity of the maping $x \rightarrow \big(\frac{x}{k}+\mathcal{C}_{kh}(v - h)(x), v(x) \big)$ we infer that the curve \eqref{eq_Surface} is non-self-intersecting while from $v(x) > 0$ we have that \eqref{eq_Surface} is contained in the upper half-plane. 
    If $\Omega$ denotes the domain whose boundary consists of $\mathcal{S}$ and $\mathcal{B}$, it follows from Theorem 3.4 in \cite{varva} that $U+iV$ is a conformal mapping from $\mathcal{R}_{kh}$ onto $\Omega$, which extends to a homeomorphism between the closures of these domains, with onto mappings
    $$\{(x,0):x \ \in\ \mathbb{R}\} \rightarrow \mathcal{S},$$
    and
    $$\{(x,-kh):x \ \in\ \mathbb{R}\} \rightarrow \mathcal{B}.$$
    Together with \eqref{UV} this implies that $\Omega$ is an L-periodic strip-like domain, with $L = 2\pi /k$. The conformal mean depth of $\Omega$ is $h$ as it can be seen from the properties of the mapping $\overline{U}+i\overline{V}:\mathcal{R}_h \rightarrow \Omega$, where $\overline{U}$,$\overline{V}$ are given by \eqref{UV_rescaling}. Let $\zeta$ be defined as the unique solution of \eqref{zeta_eq}. Then $\zeta\in  C^{2,\alpha}(\overline{\mathcal{R}_{kh}}) \cap C^\infty (\mathcal{R}_{kh})$. Now, let $\xi$ be defined by \eqref{zeta} and $S$ by \eqref{xi}. We obtain that $S$ satisfies \eqref{syst_S} along with the boundary conditions. From \eqref{eq_v} we also have that $S$ satisfies the last equation from \eqref{syst_S}. Employing now \eqref{ht_prop} we conclude that $\mathcal{S}$ is, in fact, a graph.
    \end{proof}

\section{Local bifurcation}
    This section is devoted to proving the existence of solutions to \eqref{eq_v}. The relation $[v]=h$ makes natural to set
    \begin{equation}\label{v_decomp}
    v=w+h
    \end{equation}
    Equation \eqref{v_decomp} implies immediately that $[w]=0$. We then use [19] to find that
    $$\Gkh(v)=\Gkh(w+h)=\Gkh(w) + \Gkh(h)=\frac{1}{k}+\Ckh(w')$$
    and
    $$\Gkh(v')=\Gkh(w')=\frac{[w']}{kh}+\Ckh(w'')=\Ckh(w''),$$
    since $w$ is periodic. Therefore, we can rewrite \eqref{eq_v} as
    \begin{equation}\label{eq_w}
        \begin{aligned}
            &
            \frac{\left\{ A\left(\frac{1}{k}+\Ckh(w')\right) - B w' \right\}^2 \left(w'^2+\left(\frac{1}{k}+\Ckh(w')  \right)^2\right)^{-1} }
            { A w' + B\left(\frac{1}{k}+\Ckh(w') \right)- \left(P_{atm} -\sigma\frac{\frac{w''}{k}+w''\Ckh(w')-w'\Ckh(w'')}{\left(w'^2+\left(\frac{1}{k}+\Ckh(w')  \right)^2\right)^{3/2}}\right)
            \left(w'^2+\left(\frac{1}{k}+\Ckh(w')  \right)^2\right)  }\\
            &  \\
            &+\frac{A w' + B\left(\frac{1}{k}+\Ckh(w')\right)}{  \left(w'^2+\left(\frac{1}{k}+\Ckh(w')  \right)^2\right) }
             - P_{atm} +\sigma\frac{\frac{w''}{k}+w''\Ckh(w')-w'\Ckh(w'')}{\left(w'^2+\left(\frac{1}{k}+\Ckh(w')  \right)^2\right)^{3/2}} \\
             &\\
            &=Q + 2\sigma \frac{\frac{w''}{k}+w''\Ckh(w')-w'\Ckh(w'')}{\left(w'^2+\left(\frac{1}{k}+\Ckh(w')\right)^2\right)^{3/2}}-2gh-2gw 
            \end{aligned}
            \end{equation}
            with
            \begin{equation}
            \begin{split}
            &A = P_{atm} w' -\sigma w'\frac{\frac{w''}{k} + w''\Ckh(w')-w'\Ckh(w'')}{(w'^2+(\frac{1}{k}+\Ckh(w'))^2)^{3/2}} \\
            &\\
            &B = \frac{S_0 - \sigma}{kh} + P_{atm}\left(\frac{1}{k}+\Ckh(w')\right) + g\left( \frac{[w^2]}{2kh} - \frac{w}{k} - \frac{h}{2k} + \Ckh(w w') - w \Ckh(w') \right) \\
            &+\frac{\sigma}{kh}\left[\frac{\frac{1}{k}+\Ckh(w')}{(w'^2+(\frac{1}{k} + \Ckh(w'))^2)^{1/2}}   \right] + \sigma \Ckh \left\{\frac{\Ckh(w'')w'^2 - (\frac{1}{k} + \Ckh(w'))w''w'}{(w'^2+(\frac{1}{k} + \Ckh(w'))^2)^{3/2}} \right\}, 
        \end{split}
    \end{equation}
    where, we recall, $[\cdot]$ denotes the average of the argument over one period. Moreover,
    the map $x\rightarrow\left(\frac{x}{k}+\Ckh(w)(x),w(x)+h  \right)$ is injective on $\mathbb{R}$ and
    {$w'(x)^2+\left(\frac{1}{k}+\Ckh(w')(x)\right)^2 \neq 0$} for all $x\in\mathbb{R}$
    
    We will regard $S_0$ and $Q$ as parameters and will prove the existence of solutions $w\ \in \ C_{2\pi}^{1,\alpha}$ to equation \eqref{eq_w} for all $ k>0$ and $h>0$ fixed.
    We observe that $w=0\ \in \ C_{2\pi}^{1,\alpha}$ is a solution of \eqref{eq_w} if and only if
    \begin{equation}\label{Q}
    Q = 2gh+\left(\frac{S_0}{h}- \frac{gh}{2}\right)
   \end{equation}
    This suggests setting
    \begin{equation}\label{lambda}
    \lambda = \frac{S_0}{h} - \frac{gh}{2}, 
    \end{equation}
    \begin{equation}\label{mu}
    \mu = Q - 2gh - \left(\frac{S_0}{h} - \frac{gh}{2}   \right) .
    \end{equation}
    Note that the mapping $(S_0,Q)\rightarrow (\lambda,\mu))$ is a bijection from $\mathbb{R}^2$ onto itself.
    With the notation \eqref{Q}-\eqref{mu} we have 
    \begin{equation}
    \begin{split}
& B=\frac{\lambda}{k}-\frac{\sigma}{kh}\cdot\left[\frac{w'^2}{ \left(w'^2+\left(\frac{1}{k} + \Ckh(w')\right)^2\right)^{1/2}
 \left( \frac{1}{k}+\mathcal{C}_{kh}(w')+  \left(w'^2+\left(\frac{1}{k} + \Ckh(w')\right)^2\right)^{1/2}  \right) }\right]
 \\
 &+ g\left( \frac{[w^2]}{2kh} - \frac{w}{k} + \Ckh(w w') - w \Ckh(w') \right)
+ \sigma \Ckh \left(\frac{\Ckh(w'')w'^2 - (\frac{1}{k} + \Ckh(w'))w''w'}{(w'^2+(\frac{1}{k} + \Ckh(w'))^2)^{3/2}} \right)\\
&+P_{atm}\left(\frac{1}{k}+\Ckh(w')\right).
    \end{split}
    \end{equation}
 We see that \eqref{eq_w} can be rewritten as
    \begin{equation}\label{eq_w_rew}
        \begin{aligned}
            &
            \frac{\left( A\left(\frac{1}{k}+\Ckh(w')\right) - B w' \right)^2}
            { A w' + B\left(\frac{1}{k}+\Ckh(w') \right)-\left(P_{atm} -\sigma\frac{\frac{w''}{k}+w''\Ckh(w')-w'\Ckh(w'')}{\left(w'^2+\left(\frac{1}{k}+\Ckh(w')  \right)^2\right)^{3/2}}\right)\left(w'^2+\left(\frac{1}{k}+\Ckh(w')  \right)^2\right)   }\\
            &  \\
            &+A w' + B\left(\frac{1}{k}+\Ckh(w')\right)\\
            \\
           &  - \left(P_{atm} -\sigma\frac{\frac{w''}{k}+w''\Ckh(w')-w'\Ckh(w'')}{\left(w'^2+\left(\frac{1}{k}+\Ckh(w')  \right)^2\right)^{3/2}}\right)  \left(w'^2+\left(\frac{1}{k}+\Ckh(w')  \right)^2\right) \\
             &\\
            &-\left(\lambda+\mu + 2\sigma \frac{\frac{w''}{k}+w''\Ckh(w')-w'\Ckh(w'')}{\left(w'^2+\left(\frac{1}{k}+\Ckh(w')\right)^2\right)^{3/2}}-2gw \right)\left(w'^2+\left(\frac{1}{k}+\Ckh(w')  \right)^2\right) =0
            \end{aligned}
            \end{equation}
    with $w\in C_{2\pi,o}^{1,\alpha}, \ \mu\in  \mathbb{R}$ and $\lambda\in \mathbb{R}$. 
    By means of \eqref{lambda}-\eqref{mu}
    we notice that $w=0 \in C_{2\pi,o}^{1,\alpha}$ and $\mu=0$ is a solution of \eqref{eq_w_rew} for all $\lambda \in  \mathbb{R}$. We now apply the Crandall-Rabinowitz Theorem [34] on bifurcation from simple eigenvalues in order to prove the existence of non-trivial solutions to \eqref{eq_w_rew}.

    \begin{thm}\label{CranRab}
        Let $\mathbb{X}$ and $\mathbb{Y}$ be Banach spaces, I an open interval in $\mathbb{R}$ containing $\lambda^*$, and $\mathcal{F}: I \times \mathbb{X} \rightarrow \mathbb{Y}$ be a continuous map satisfying the following properties:
        \begin{enumerate}
            \item $\mathcal{F}(\lambda,0)=0\,\,{\rm for}\,\,{\rm all}\,\,\lambda\in\ I$;
            \item $\partial_\lambda \mathcal{F}, \, \partial_u \mathcal{F}\ {\rm and}\ \partial^2_{\lambda,u}\mathcal{F}$ exist and are continuous;
            \item $\mathcal{N}(\partial_u \mathcal{F}(\lambda^*,0))$ and $\mathbb{Y}/\mathcal{R}(\partial_u \mathcal{F}(\lambda^*,0))$ are one dimensional, with the null space generated by $u^*$;
            \item $\partial^2_{\lambda,u}\mathcal{F}(\lambda^*,0)(1,u^*) \notin \mathcal{R}(\partial_u \mathcal{F}(\lambda^*,0))$.
        \end{enumerate}
        Then there exists a continuous local bifurcation curve $\{\lambda(s), u(s):|s| < \varepsilon \}\ with\ \varepsilon>0$ sufficiently small such that $(\lambda(0),u(0)) = (\lambda^*,0)$ and there exists a neighborhood $\mathcal{O}$ of $(\lambda^*,0) \in I \times \mathbb{X}$ such that \\
        $\{(\lambda,u) \in \mathcal{O} : u \neq 0, \mathcal{F}(\lambda,u)=0\} = \{\lambda(s), u(s): 0 < |s|<\varepsilon\}$\\
        Moreover, we have \\
        $u(s) = su^*+o(s)\ in\ \mathbb{X},|s|<\varepsilon$ \\
        If $\partial_u^2\mathcal{F}$ is also continuous, then the curve is of class $C^1$.
        \end{thm}
        
        In order to apply Theorem \ref{CranRab} to \eqref{eq_w_rew} we set 
        $$\mathbb{X} = \mathbb{R} \times C^{p+1,\alpha}_{2\pi,o,e},\quad \mathbb{Y} = C^{p,\alpha}_{2\pi,e}, $$
        where for any integer $p\geq0$ we denote:\\
        $C^{p,\alpha}_{2\pi,e} = \{f \in C^{p,\alpha}_{2\pi}:f(x) = f(-x)\,\, {\rm for}\,\,{\rm all}\,\,x\in \mathbb{R}\}$ \\
        $C^{p,\alpha}_{2\pi,o,e} = \{f \in C^{p,\alpha}_{2\pi,o}:f(x) = f(-x)\,\, {\rm for}\,\,{\rm all}\,\, x \in \mathbb{R}\}$
    
    Equation \eqref{eq_w_rew} can be written as $\mathcal{F}(\lambda,(\mu,w))=0$ where
     $\mathcal{F}:\mathbb{R} \times \mathbb{X} \rightarrow \mathbb{Y}$ is given by the expression on the left hand side of \eqref{eq_w_rew}.
    Since $\mathcal{F}(\lambda,(0,0))=0$, the first condition on the local bifurcation theorem is verified. We now compute the derivative of $\mathcal{F}$ with respect to the $(\mu, w)$ variable, which is defined through
    $$\partial_{\mu,w}\mathcal{F}(\lambda,(0,0))(\nu,f) = \lim_{t\to 0} \frac{\mathcal{F}(\lambda,t(\nu,f)) - \mathcal{F}(\lambda,(0,0))}{t}$$
    Using Lemma \ref{ht_prop},  we have
    \begin{equation}\label{deriv_F}
        \begin{aligned}
            \partial_{\mu,w}\mathcal{F}(\lambda,(0,0))(\nu,f) &= -\frac{\lambda}{k} \Ckh(f') +\frac{gf}{k^2}
            -\sigma'' - \frac{\nu}{k^2}\\  
            &= -\frac{1}{k^2}(\lambda k \Ckh(f') +\sigma k^2 f'' - gf) -\frac{\nu}{k^2} 
        \end{aligned}
    \end{equation}
    Using Lemma \ref{ht_rep} we find that 
    \begin{equation}
        \partial_{\mu,u}\mathcal{F}(\lambda,(0,0))(\nu,f) = - \frac{1}{k^2} 
         \sum_{n=1}^{\infty} (\lambda (k n) \coth (nkh) 
         - \sigma k^2 n^2 - g) a_n \cos(nx) - \frac{\nu}{k^2} 
        \end{equation}
     if $f=\sum_{n=1}^{\infty}a_n \cos(nx)$.
    
    Now, using Lemma \ref{ht_prop} it follows that the bounded linear operator $\partial_{(\mu,u)} \mathcal{F}(\lambda,(0,0)):\mathbb{X}\rightarrow \mathbb{Y}$ is invertible whenever 
    $$\lambda (kn) \coth (nkh) - \sigma k^2 n^2 - g \neq 0$$
     for any integer $n\geq1$. Therefore, for bifurcating solutions we get,
    \begin{equation}\label{eq_lambda}
        \lambda (kn) \coth(nkh) - \sigma k^2 n^2 - g = 0 
    \end{equation}
    for some integer $n\geq1$. Since, we are looking for solutions of \eqref{eq_lambda} of minimal period $2\pi$, we can take $n=1$ in \eqref{eq_lambda}.
    
    Let
    $$\lambda^{*n} = \left(\sigma kn + \frac{g}{kn}\right) \tanh(nkh)$$
    denote the solution of \eqref{eq_lambda}.
    
   \begin{remark}\label{1to_n}
    Note that $\lambda^{*1} (nk) = \lambda^{*n}(k)$, for all integers
   $n \ge 1$ and all $k > 0$.
   \end{remark}
    \begin{lemma}\label{strict_inc}
        Let $\lambda (k) = \lambda^{*1} (k)$. If $\frac{\sigma}{gh^2} > \frac{1}{3}$ then the function $\lambda$ is strictly increasing.
    \end{lemma}
    \begin{proof}
    We remark that, denoting $f(k):=k\coth(kh)$, we have 
     \begin{equation}\label{f_lambda}
     \lambda(k)=\frac{k^2\sigma +g}{f}.
     \end{equation}
    It is evident that $\lim_{k\to \infty} \lambda(k) = \infty$. On the other hand, we note that if $k_0>0$ is a critical point of 
    $\lambda$ then 
    $$\lambda(k_0) f'(k_0)=2k_0\sigma.$$
    Furthermore, 
    \begin{equation}
    f(k_0)\lambda''(k_0)+f''(k_0)\lambda(k_0)=2\sigma,
    \end{equation}
   that is,
    \begin{equation}\label{f_lambdadprime}
    f(k_0)\lambda''(k_0)= 2\sigma\frac{f'(k_0) - k_0f''(k_0}{f'(k_0)} 
   \end{equation}
    where $$f'(k) - kf''(k) = \frac{ \sinh^2(hk) \cosh(hk) + hk \sinh(hk)-2(hk)^2 \cosh(hk)}{\sinh^3(hk)}\ \ .$$
    Denoting $g(x) = \sinh^2 x \cosh{x} + x\sinh{x}-2x^2\cosh{x}$, we obtain $g^{(3)}(x) > 0$ for all $ x>0$, $g(0)=g'(0)=g''(0)$ which implies that $g(x)>0$ for all $x>0$. The latter result along with \eqref{f_lambdadprime} implies  $\lambda''(k_0)>0$, since $f'(k_0)>0$. Hence, $\lambda$ achieves a local minimum at any critical point $k_0>0$.
     The condition in the statement of the theorem is equivalent with $\lambda''(0)>0$ which renders $k=0$ to a point where $\lambda$ achieves a local minimum. Adding that $\lim_{k \to \infty} \lambda(k) = \infty$ and the fact that $\lambda$ cannot have a local maxima at points $k>0$ proves the theorem.
    \end{proof}
     Lemma \ref{strict_inc} and Remark \ref{1to_n} give the necessary and sufficient condition for the one-dimensionality of the kernel $\mathcal{N}(\partial_{\mu,w}\mathcal{F}(\lambda^*,0))$. We state this condition in the following lemma.
    
    \begin{lemma}
        Let $\lambda^{*} = (k\sigma + \frac{g}{k}) \tanh{kh}$ be the solution to \eqref{eq_lambda}. Then the kernel $\mathcal{N}(\partial_{\mu,w}\mathcal{F}(\lambda^*,0))$ is one-dimensional if and only if 
        $$\frac{\sigma}{gh^2} \geq \frac{1}{3}$$
        Moreover, in this situation $\mathcal{N}(\partial_{\mu,w}\mathcal{F}(\lambda^*,0))$ is generated by $(0,w^*) \in \mathbb{X}$, where $w^{*}(x)=\cos(x)$.
    \end{lemma}
    We now give a sufficient condition that ensures the one dimensionality of the kernel $\mathcal{N}(\partial_{\mu,w}\mathcal{F}(\lambda^*,0))$. This condition concerns capillary-gravity waves with the wavelength not exceeding $2$ cm, and is reflected in the requirement on the wave number $k$ from the statement of the following lemma.
    \begin{lemma}
        Let $\lambda^{*1}$ be the solution of \eqref{eq_lambda} with $n=1$, i.e.
        $$\lambda^{*1} = \left(k\sigma + \frac{g}{k}  \right)\tanh{kh}$$
        Assume that $k>\pi \cdot 10^{-2}\,m^{-1}$. Then the kernel $\mathcal{N}(\partial_{\mu,w}\mathcal{F}(\lambda^*,0))$ is one dimensional being generated by $(0,w^*) \in \mathbb{X}$, where $w^*(x) = \cos x$ for all $x \in \mathbb{R}$
    \end{lemma}
    \begin{proof}It suffices to show that if $n\in \mathbb{N}, n>1$ then the equations
    \begin{equation}\label{equations}
        \begin{aligned}
            \lambda kn \coth(nkh) - \sigma k^2n^2 - g &= 0 \\
            \lambda k \coth(kh) - \sigma k^2 - g &= 0
        \end{aligned}
    \end{equation}
    do not have common solutions.
    We assume $\lambda^{*n} = \lambda^{*1}$ for some $n>1$ which leads to 
    \begin{equation*}
        \left(nk\sigma + \frac{g}{kn}  \right) \tanh(nkh) = \left(k \sigma + \frac{g}{k}\right)\tanh(kh),
    \end{equation*}
   written also as 
    \begin{equation*}
        \frac{\tanh(nkh)}{n} - \tanh(kh) = \frac{k^2 \sigma}{g} \left(\tanh(kh) - n \tanh(nkh)  \right),
    \end{equation*}
    equation that is equivalent with 
   \begin{equation}
    F(n) - F(1) = C(G(1) - G(n)),
    \end{equation}
   if we set 
   $$F(n) = \frac{\tanh(nkh)}{n}, \quad G(n)=n\tanh(nkh),\quad C=\frac{k^2\sigma}{g}.$$
  Our aim now will be to prove that 
    $$ E(n) := C(G(n) - G(1)  + F(n) - F(1)$$ is strictly positive for all $n\in\mathbb{N}$ with $n\geq 2$. To this end we notice that for capillary-gravity water waves with wavelength not exceeding $2$ cm, we obtain through a simple calculation
    that $C>\frac{1}{2}$. Thus, we infer that 
    \begin{equation}
    \begin{split}
    E(n)>&\frac{G(n) - G(1)}{2}+F(n) - F(1)\\
    &=\frac{n\tanh(nkh)-3\tanh(kh)}{2}+\frac{\tanh(nkh)}{n}
    \end{split}
    \end{equation}
    For $n\geq 3$ we clearly have that $n\tanh(nkh)-3\tanh(kh)>0$, which implies that $E(n)>0$ for all $n\geq 3$. On the other hand, for $n=2$ we see that $E(2)=\frac{3}{2}(\tanh(2kh)-\tanh(kh)),$ expression that is also strictly positive. Hence, the claim that equations \eqref{equations} have no common solutions is proved.
    \end{proof}
    We are left with the examination of the transversality condition from the Crandall-Rabinowitz Theorem. 
    It is easy to see that $\mathcal{R}(\partial_{(\mu,w)}F(\lambda^*,0))$ is the closed subspace of $\mathbb{Y}$ consisting of all functions $f \in \mathbb{Y}$ which satisfy
    $$\int_{-\pi}^{\pi}f(x)\cos(x)dx=0 ,$$
    and therefore $\mathbb{Y}/\mathcal{R}(\partial_{(\mu,w)}F(\lambda^*,0))$ is the one-dimensional subspace of $\mathbb{Y}$ generated by the function $w^*(x) = \cos(x)$. Using \eqref{deriv_F} we compute
    \begin{equation*}
        \partial^2_{\lambda,(\mu.w)}F(\lambda*,(0,0))(1,(0,w^*)) = \lim_{t \to 0}\frac{\partial_{(\mu,w)}F(\lambda^*+t,(0,0))(0,w^*) - \partial_{(\mu,w)}F(\lambda^*,(0,0))(0,w^*)}{t}
    \end{equation*}
 obtaining that
    \begin{equation*}
        \partial^2_{\lambda,(\mu,w)}F(\lambda^*,(0,0))(1,(0,w^*))=-\frac{1}{k}\lambda^{*} \Ckh(w^*{}') 
    \end{equation*}
    For $w^* = \cos{x}$ we have from Lemma \ref{ht_prop} that $\Ckh(w^*) = \coth{(kh)} \sin x$ and $\Ckh(w^*{}')$ = $(\Ckh(w^*))'$ = $\coth{(kh)} \cos x $ = $\coth(kh)w^*$, and therefore,
    \begin{equation*}
        \partial^2_{\lambda,(\mu,w)}F(\lambda^*,(0,0))(1,(0,w^*))=-\frac{1}{k} \coth(kh) w^* \notin \mathcal{R}(\partial_{(\mu,w)}F(\lambda^*,(0,0))) 
    \end{equation*}
    since by \eqref{f_lambda} we have
    $$-\frac{\lambda^*}{k}\coth (kh)\int_{-\pi}^{\pi}w^*(x)\cos x \,dx = -\pi\left(\sigma+\frac{g}{k^2}\right)\neq 0 .$$
    From the local bifurcation theorem we obtain the bifurcation values
    \begin{equation}\label{bif_values}
    \lambda = \frac{k^2\sigma+g}{k} \tanh{(kh)}
    \end{equation}
    Note that if we set $\sigma=0$ in \eqref{bif_values}, we recover the well-known dispersion relation for irrotational gravity water waves over flows of  finite depth, cf. \cite{Din}.
  Moreover,  $g=0$ in the previous formula retrieves the dispersion relation for irrotational capillary waves, cf. \cite{Phil}.
    
    We obtain the corresponding values for the flow force $S_0$ as
    \begin{equation}\label{flowf_values}
        S_0 =  \frac{k^2 \sigma h +g h}{k}\tanh(kh)+ \frac{gh^2}{2}.
    \end{equation}
    In the case of the gravity waves $(g=0)$ formula \eqref{flowf_values} recovers formula (5.14) in \cite{Basu2019JDE}.
   
    We can now formulate the results concerning the existence of small amplitude periodic capillary-gravity irrotational water waves.
    \begin{thm}
        For any $h>0,\,k>0$ and $S_{0}\in \mathbb{R}$ satisfying $k>\pi\cdot 10^{-2}$ there exists laminar flows with a flat free surface in water of depth h and modified flow force at the surface $S_0$. The laminar flows with modified flow force $S_0$ at the surface are exactly those with horizontal speeds at the flat free surface equal to $\sqrt{\lambda}$ given by \eqref{bif_values}. The values of $S_0$ of the flow force given by \eqref{flowf_values} trigger the appearance of periodic steady waves of small amplitude, with period $\frac{2\pi}{k}$ and conformal mean depth h, which have a smooth profile with one crest and one trough per period, monotone between consecutive crests and troughs and symmetric about any crest line.
    \end{thm}
    \begin{proof}Using the argument from the proof of Theorem 3.2 in [22] we see that $w=0$ gives rise to laminar flows in the fluid domain bounded below by the rigid bed $\mathcal{B}$ and above by the flat free surface $Y=h$. 
    These laminar flows are given by 
    $$S(X,Y) = -\frac{g}{2}Y^2 + \left(\frac{S_0}{h} + \frac{g h}{2}\right)Y , \quad X \in \mathbb{R}, \,0 \leq Y \leq h.$$
Observe that $S_{Y|Y=h} = (c-u_{1})^{2}|_{Y = h}= \frac{S_0}{h} - \frac{g h}{2}$ which shows that for laminar flows the horizontal velocity at the free surface coincides with $\sqrt{\lambda}$ given in \eqref{lambda}. Formula \eqref{bif_values}, detailing the speed $\sqrt{\lambda}$ at the free surface in terms of the depth $h$ and period $2\pi/k$, is called the \textit{dispersion relation}.
    
    Concerning the existence of waves of small amplitude with the properties mentioned in the statement of the theorem we apply the Crandall-Rabinowitz Theorem which asserts the existence of the local bifurcation curve
    $$\{(\lambda(s),(0+o(s),s\cos{x} + o(s))):|s|<\varepsilon  \} \subset \mathbb{R} \times \mathbb{X}$$
    consisting of the solutions of \eqref{eq_w_rew} with $\lambda$ given by \eqref{bif_values}.
    
    Choosing $\varepsilon$ sufficiently small and using Lemma \ref{ht_prop} we can ensure that 
    $$w(x) > -h\,\,{\rm for}\,\,{\rm all}\,\, x \in \mathbb{R},$$
    and
    $$\frac{1}{k} + \Ckh(w')(x)>0\,\,{\rm for}\,\,{\rm all}\,\, x \in \mathbb{R}.$$
    The above inequality implies that the corresponding non-flat free surface $\mathcal{S}$ given by \eqref{surface_form} with $v=w+h$ is the graph of a smooth function, 
    symmetric with respect to the points obtained for the values $x=n\pi,n \in \mathbb{Z}$. From 
    $$w(x;s) = s\cos{x} + o(s)\,\,{\rm in}\,\, C_{2\pi}^{p+1,\alpha},$$
    we have that 
    $sw'(x;s) < 0$ for all $x \in (0,\pi), 0 <|s|<\varepsilon,$
    for $\varepsilon > 0$ sufficiently small and $p \geq 1.$ Using the evenness of $x \to w(x;s)$ we conclude the proof of the assertion about the free surface $\mathcal{S}$, i.e. $\mathcal{S}$ has one crest and one trough per minimal period and is monotone between consecutive crests and troughs.
    \end{proof}
    \noindent {\bf Acknowledgements}: C. I. Martin would like to acknowledge the support of the Austrian Science Fund through research grant P 30878-N32.


\begin{thebibliography}{ll}
    \bibitem{stokes1847} G. Stokes, On the theory of oscillatory waves, \emph{Trans. Cambridge Phil. Soc.} 8 (1847)
441–455.

\bibitem{nekrasov1921steady}  A. I. Nekrasov, On steady waves, Izv. Ivanovo-Voznesenk. Politekhn. In-ta 3 (1921)
52–65.

\bibitem{levi1924determinazione} T. Levi-Civita, Determinazione rigorosa delle onde irrotazionali periodiche in acqua profonda, \emph{Rend. Acad. Lincei.}  33 (1924) 141–150.

\bibitem{struik1926determination} D. J. Struik, D{\'e}termination rigoureuse des ondes irrotationelles p{\'e}riodiques dans un canal {\`a} profondeur finie, \emph{Mathematische Annalen} 95 (1) (1926) 595–634.

\bibitem{krasovskii1961theory} Y. P. Krasovskii, On the theory of steady-state waves of finite amplitude, \emph{U.S.S.R.
Comp. Math. and Math. Phys.} 1 (1961) 996–1618.

\bibitem{toland1996stokes} J. F. Toland, Stokes waves, \emph{Topol. Methods Nonlinear Anal.} 7 (1) (1996) 1–48.

\bibitem{mcleod1980stokes}  J. B. McLeod,  Stokes and Krasovskii’s conjectures for the wave of greatest height, \emph{Stud. Appl. Math}. 98 (4) (1980) 311–334.

\bibitem{constantin2004exact} A. Constantin, W. Strauss, Exact steady periodic water waves with vorticity, \emph{Comm.
Pure Appl. Math.} 57 (4) (2004) 481–527.

\bibitem{constantin2011periodic}  A. Constantin, W. Strauss, Periodic traveling gravity water waves with discontinuous vorticity, \emph{Arch. Ration. Mech. Anal}. 202 (1) (2011) 133–175.

\bibitem{walsh2009steady}  S.  Walsh,  Steady  stratified  periodic  gravity  waves  with  surface  tension  II:  Global bifurcation, \emph{Discr. Cont. Dyn. Sys. Ser. A} (8) (2014) 3241 – 3285.

\bibitem{walsh2009stratified} S. Walsh, Stratified steady periodic water waves, \emph{SIAM J. Math. Anal.}  41 (3) (2009)
1054–1105.

\bibitem{ac11} A. Constantin, \emph{Nonlinear Water Waves with Applications to Wave-Current Interactions and Tsunamis}, CBMS-NSF Conf. Series Appl. Math., Vol. 81, SIAM, Philadel- phia, 2011.

\bibitem{conswahlen07}  A. Constantin, M. Ehrnstr{\"o}m, E. Wahl{\'e}n, Symmetry of steady periodic gravity water waves with vorticity, \emph{Duke Math. J.} 140 (2007) 591–603.

\bibitem{constantin2004symmetry}  A.  Constantin,  J.  Escher,  Symmetry  of  steady  periodic  surface  water  waves  with vorticity,  \emph{J. Fluid Mech}. 498 (2004) 171–181.

\bibitem{constantin2006variational}  A. Constantin, D. Sattinger, W. Strauss, Variational formulations for steady water waves with vorticity, \emph{J. Fluid Mech.}  548 (2006) 151–163.

\bibitem{constantin2007stability}   A. Constantin, W. A. Strauss,  Stability properties of steady water waves with vorticity, \emph{Comm. Pure Appl. Maths.} 60 (6) (2007) 911–950.

\bibitem{constantin2004symdeep}  A. Constantin, J. Escher, Symmetry of steady deep-water waves with vorticity,  \emph{Eur. J. Appl. Math}. 15 (6) (2004) 755–768.

\bibitem{ac11am}  A. Constantin, J. Escher, Analyticity of periodic travelling free surface waves with vorticity, \emph{Annals of Mathematics} 173 (2011) 559–568.

\bibitem{cons_varvaruca2011} A. Constantin,  E. Varvaruca,  Steady periodic water waves with constant vorticity: Regularity and local bifurcation,  \emph{Arch. Ration. Mech. Anal.} 199 (2011) 33–67.

\bibitem{benjamin}  T. B. Benjamin, M. J. Lighthill, On cnoidal waves and bores,  \emph{Proc. R. Soc. London A} 224 (1954) 448–460.

\bibitem{keadynorbury78b}  G. Keady, J. Norbury, Waves and conjugate streams with vorticity, \emph{Mathematika} 25 (1978) 129–150.

\bibitem{kozlov} V.  Kozlov,  N.  Kuznetsov,  E.  Lokharu,  On  the  Benjamin–Lighthill  conjecture  for water waves with vorticity, \emph{J. Fluid Mech.} 825 (2017) 961–1001.

\bibitem{Basu2019JDE}  B.  Basu,  A flow force  reformulation  for  steady  irrotational  water  waves,  \emph{J.  Diff.
Equations.} 268 (12) (2020) 7417–7452.

\bibitem{HenCG} D.  Henry,  Analyticity of the free surface  for  periodic  travelling  capillary-gravity water waves with vorticity, \emph{Journal of Mathematical Fluid Mechanics} 14 (2) (2012)
249–254.

\bibitem{henrymatioc} D. Henry, B.-V. Matioc, On the existence of steady periodic capillary-gravity stratified water waves, \emph{Ann. Scuola Norm. Sup. Pisa} 12 (2013) 955–974.

\bibitem{martin2013} C.  I.  Martin,  Local  bifurcation  and  regularity  for  steady  periodic  capillary-gravity water  waves  with  constant  vorticity,  \emph{Nonlinear  Analysis:  Real  World  Applications} 14 (2013) 131–149.

\bibitem{MarMatSJAM}  C. Martin, B.-V. Matioc, Existence of wilton ripples for water waves with constant vorticity and capillary effects, \emph{SIAM Journal on Applied Mathematics} 73 (4) (2013) 1582–1595.

\bibitem{MarMatJNS} C.  Martin,  B.-V.  Matioc,  Steady  periodic  water  waves  with  unbounded  vorticity: Equivalent  formulations  and  existence  results,  \emph{Journal  of  Nonlinear  Science } 24  (4) (2014) 633–659.

\bibitem{MarMatJDE} C. Martin, B.-V. Matioc, Existence of capillary-gravity water waves with piecewise constant vorticity, \emph{Journal of Differential Equations} 256 (8) (2014) 3086–3114.

\bibitem{MatMfM} B.-V. Matioc, Global bifurcation for water waves with capillary effects and constant vorticity, \emph{Monatshefte fur Mathematik} 174 (3) (2014) 459–475.

\bibitem{MatMatCMP} A.-V. Matioc, B.-V. Matioc, Capillary-gravity water waves with discontinuous vor- ticity:  Existence  and  regularity  results,  \emph{Communications  in  Mathematical  Physics}. 330 (2) (2014) 859–886.

\bibitem{wahlen2006}  E.  Wahl{\'e}n,  Steady  periodic  capillary-gravity  waves  with  vorticity,  \emph{SIAM  J.  Math. Anal.} 38(3) (2006) 921–943.

\bibitem{varva} E. Varvaruca, Bernoulli free-boundary problems in strip-like domains and a property of permanent waves on water of finite depth, \emph{Proc. Roy. Soc. of Edinburgh Sect. A} 138 (2008) 1345–1362.

\bibitem{Din}  M. W. Dingemans, \emph{Water wave propagation over uneven bottoms}, NASA Sti/Recon
Technical Report N, Advanced Series on Ocean Engineering, vol. 13, 1997.

\bibitem{Phil} O. Phillips, \emph{The dynamics of the upper ocean }(2nd ed.), Cambridge University Press,
1977.





\end{thebibliography}
\end{document}